\documentclass[10pt,a4paper,english,reqno]{amsart}
\usepackage[T1]{fontenc}
\usepackage[utf8]{inputenc}
\usepackage{babel}
\usepackage{amsthm}
\usepackage{amstext}
\usepackage{amssymb}
\usepackage{esint}
\usepackage{enumerate}
\usepackage[unicode=true,pdfusetitle,
 bookmarks=true,bookmarksnumbered=false,bookmarksopen=false,
 breaklinks=false,pdfborder={0 0 1},backref=section,colorlinks=false]
 {hyperref}

\makeatletter

\pdfpageheight\paperheight
\pdfpagewidth\paperwidth

\theoremstyle{plain}
\newtheorem{thm}{Theorem}[section] 
\theoremstyle{definition}

\theoremstyle{plain}

\theoremstyle{remark}
\newtheorem{rem}[thm]{\protect\remarkname}
\theoremstyle{plain}
\newtheorem{lem}[thm]{\protect\lemmaname}
\theoremstyle{plain}
\newtheorem{cor}[thm]{\protect\corollaryname}
\theoremstyle{plain}
\newtheorem{ex}[thm]{\protect\examplename}

\usepackage{a4wide}
\usepackage{bbm}
\usepackage[arrow, matrix, curve]{xy}
\usepackage{tikz}

\renewcommand{\1}{\mathbbm{1}}
\renewcommand{\tilde}{\widetilde}

\renewcommand{\phi}{\varphi}
\renewcommand{\bar}{\overline}
\usepackage{amssymb}
\usepackage{mathtools}

\makeatother

  \providecommand{\corollaryname}{Corollary}
  \providecommand{\definitionname}{Definition}
  \providecommand{\lemmaname}{Lemma}
  \providecommand{\remarkname}{Remark}
  
  \providecommand{\examplename}{Example}

\begin{document}
\title[Strong laws of large numbers for intermediate trimming]{ Strong laws of large numbers for intermediately trimmed sums of
i.i.d.\ random variables with infinite mean}

\author[Kesseb\"ohmer]{Marc Kesseb\"ohmer}
  \address{Universit\"at Bremen, Fachbereich 3 -- Mathematik und Informatik, Bibliothekstr. 1, 28359 Bremen, Germany}
  \email{\href{mailto:mhk@math.uni-bremen.de}{mhk@math.uni-bremen.de}}
\author[Schindler]{Tanja Schindler}
\address{Australian National University, College of Business and Economics, 
Acton ACT 2601, Australia}
  \email{\href{mailto:tanja.schindler@anu.edu.au}{tanja.schindler@anu.edu.au}}

\date{\today}

\thanks{This research was supported  by the German Research Foundation (DFG)
grant \emph{Renewal Theory and Statistics of Rare Events in Infinite
Ergodic Theory} (Geschäftszeichen KE 1440/2-1). TS was supported by
the Studienstiftung des Deutschen Volkes.}

\begin{abstract}
{ We show that for every sequence of  non-ne\-ga\-tive i.i.d.\ random
variables with infinite mean there exists
a proper moderate trimming such that for the trimmed sum process a non-trivial
strong law of large numbers holds. We provide an explicit procedure  to find  a moderate trimming sequence 
even if the underlying distribution function has a complicated structure, e.g. has no regularly varying tail distribution.}
\end{abstract}
\maketitle

\section{Introduction and statement of main results}\label{sec: introduction}
Throughout the paper, let $\left(X_{n}\right)_{n\in\mathbb{N}}$ denote
a sequence of non-negative, independent and identically distributed
(i.i.d) random variables with distribution function $F:\mathbb{R}\to\left[0,1\right]$,
$x\mapsto\mathbb{P}\left(X_{1}\leq x\right)$ and define the sum process
$S_{n}\coloneqq\sum_{k=1}^{n}X_{k}$ for $n\geq1$ and $S_{0}\coloneqq0$.
We say $\left(X_{n}\right)_{n\in\mathbb{N}}$ fulfills a strong law
of large numbers (with normalizing sequence $\left(d_{n}\right)$)
if $\lim S_{n}/d_{n}=1$ almost everywhere (a.e.). In contrast to
the case of finite expectation, if $\mathbb{E}\left(X_{1}\right)=\infty$
then $\left(X_{n}\right)$ fulfills no strong law of large numbers.
In fact, Aaronson showed in \cite{aaronson_ergodic_1977} that for
all positive sequences $\left(d_{n}\right)_{n\in\mathbb{N}}$ we have
almost everywhere that 
\[
\limsup_{n\rightarrow\infty}\frac{S_{n}}{d_{n}}=+\infty\text{ \,\,\,\,\ or \,\,\,\,\,}\liminf_{n\rightarrow\infty}\frac{S_{n}}{d_{n}}=0.
\]
However, if there is a sequence of constants $\left(d_{n}\right)_{n\in\mathbb{N}}$
such that $\lim_{n\rightarrow\infty}S_{n}/d_{n}=1$ in pro\-ba\-bi\-li\-ty,
then there might be a strong law of large numbers after deleting finitely
many of the largest summands from the partial $n$-sums. More precisely,
for each $n\in\mathbb{N}$ we choose a permutation $\sigma$ from the symmetric group  $\mathcal{S}_{n}$ acting on $\{1,\ldots,n\}$
  such that  $X_{\sigma\left(1\right)}\geq X_{\sigma\left(2\right)}\geq\ldots\geq X_{\sigma\left(n\right)}$. For given $b_{n}\in\mathbb{N}_{0}$ we then set
\begin{align}
S_{n}^{b_{n}} & \coloneqq\sum_{k=b_n+1}^{n}X_{\sigma\left(k\right)}.\label{Snf(n)}
\end{align}
If $b_{n}=r\in\mathbb{N}$ is fixed for all $n\in\mathbb{N}$ then $\left(S_{n}^{r}\right)$
is called a \emph{lightly trimmed sum} \emph{process}.

For an example of this situation we consider the unique continued
fraction expansion of an irrational $x\in\left[0,1\right]$ given by
\[
x\coloneqq[a_{1}\left(x\right),a_{2}\left(x\right),\ldots]\coloneqq\frac{1}{a_{1}\left(x\right)+\cfrac{1}{a_{2}\left(x\right)+\ddots}}
\]
 Then $X_{n}\coloneqq a_{n}$, $n\in\mathbb{N}$, defines almost everywhere
a stationary (dependent, but $\psi$-mixing) process with respect
to the Gauss measure $\mathrm{d}\mathbb{P}(x)\coloneqq1/\left(\log2\left(1+x\right)\right)\mathrm{d}\lambda\left(x\right)$,
where $\lambda$ denotes the Lebesgue measure restricted to $\left[0,1\right]$.
Khinchin showed in \cite{khintchine_metrische_1935} that for the
normalized sum of the continued fraction digits we have $\lim_{n\rightarrow\infty}S_{n}/\left(n\log n\right)=1/\log2$
in probability. Even though a strong law of large numbers can not
hold for $S_{n}$, Diamond and Vaaler showed in \cite{diamond_estimates_1986}
that under light trimming with $r=1$ we have Lebesgue almost everywhere
\[
\lim_{n\rightarrow\infty}\frac{S_{n}^{1}}{n\log n}=\frac{1}{\log2}.
\]
 We refer to this as a {\em lightly trimmed strong law}.
  Mori provided in \cite{mori_strong_1976}, \cite{mori_stability_1977}
for i.i.d.\ random variables general conditions 
on the distribution function for a lightly trimmed strong law to hold.
These results have been
generalized by Kesten and Maller, see \cite{maller_relative_1984},
\cite{kesten_ratios_1992}, and \cite{kesten_effect_1995}, see also
\cite{einmahl_relationship_1988} and \cite{csoergoe_strong_1996} for further results. 
Aaronson
and Nakada extended the results of Mori to $\psi$-mixing random variables
in \cite{aaronson_trimmed_2003}.

The above results show that for certain classes of distribution functions
we can obtain almost sure limit theorems under  light trimming. 
However, { a theorem by Kesten in \cite{kesten_convergence_1993} states that for any fixed $r\in\mathbb{N}$ and sequences  $\left(a_{n}\right)_{n\in\mathbb{N}}$
and $\left(d_{n}\right)_{n\in\mathbb{N}}$ with $d_{n}\rightarrow\infty$ the convergence in distribution of $\left(S_{n}-a_{n}\right)/d_{n}$ is equivalent to the convergence in distribution of
$\left(S_{n}^{r}-a_{n}\right)/d_{n}$.}
Hence, this theorem shows that a weak law of large numbers for $S_{n}$ is
ne\-ces\-sa\-ry for a lightly trimmed strong law.
 Combining   two theorems by Feller { \cite[VII.7 Theorem 2]{feller_introduction_1971} and \cite[VIII.9 Theorem 1]{feller_introduction_1971}} shows that 
for functions with a distribution function with regularly
varying tails with exponent larger than $-1$, i.e.\ $1-F\left(x\right)\sim x^{-\alpha}L\left(x\right)$
with $0<\alpha<1$ and $L$ a slowly varying function, there is no weak law of large numbers and hence there is no lightly trimmed strong law.
Inhere, $u\left(x\right)\sim w\left(x\right)$ means that $u$ is
asymptotic to $w$ at infinity, that is $\lim_{x\rightarrow\infty}u\left(x\right)/w\left(x\right)=1$
and $L$ being \emph{slowly varying} means that for every $c>0$ we
have $L\left(cx\right)\sim L\left(x\right)$. 
In particular this shows the need of a stronger version of trimming than light trimming: Instead of considering the trimming by a constant $b_{n}=r$
in \eqref{Snf(n)} we allow for a sequence $\left(b_{n}\right)\in\mathbb{N}^{\mathbb{N}}$
diverging to infinity such that and $b_{n}=o\left(n\right)$, i.e.\ $\lim_{n\rightarrow\infty}b_{n}/n=0$, and
we then consider the \emph{intermediately} (also called \emph{moderately})  \emph{trimmed sum
process} $\left(S_{n}^{b_{n}}\right)$.

The case of regularly varying tails is treated by Haeusler and Mason in \cite{haeusler_laws_1987} and Haeusler in \cite{haeusler_nonstandard_1993}, in which a law of  an iterated logarithm is established.
With $F^{\leftarrow}:\left[0,1\right]\to\mathbb{R}_{\geq 0}$ denoting the  \emph{generalized inverse
function} of $F$, i.e.\ 
$F^{\leftarrow}\left(y\right) \coloneqq\inf\left\{ x\in\mathbb{R}\colon F\left(x\right)\geq y\right\}$,
they proved that 
\begin{align*}
\limsup_{n\to\infty} \pm \frac{S_n^{b_n}-n\cdot \int_0^{1-b_n/n}F^{\leftarrow}\left(s\right)\mathrm{d}s}{\gamma\left(n,b_n\right)}
\end{align*}
almost surely equals $1$ if $\lim_{n\to\infty}b_n/\log\log n=\infty$, see \cite{haeusler_laws_1987}, and, if $b_n\sim c\cdot \log\log n$, almost surely equals a constant $M$,  see \cite{haeusler_nonstandard_1993}.
By comparing the asymptotic behavior of the norming and centering sequences $\gamma\left(n,b_n\right)$ and $n\cdot\int_0^{1-b_n/n}F^{\leftarrow}\left(s\right)\mathrm{d}s$ referring to \cite[Section 4]{haeusler_nonstandard_1993} one can conclude that $\lim_{n\to\infty}S_n^{b_n}/(n\cdot \int_0^{1-b_n/n}F^{\leftarrow}\left(s\right)\mathrm{d}s)=1$ almost surely
if and only if $\lim_{n\to\infty}b_n/\log\log n=\infty$. We refer to this behavior, i.e.\ the existence of a sequence $\left(d_n\right)$ such that $\lim_{n\to\infty}S_n^{b_n}/d_n=1$ almost surely, in the following as an \emph{intermediately trimmed strong law}.

Even though for regularly varying tail distributions  an intermediately trimmed strong law can be derived from the above results, there is  little known for general distribution functions.
For slowly varying tail distributions Haeusler and Mason provided in \cite{haeuler_asymptotic_1991}  a condition  depending on the distribution of the $b_n$-th maximal term  for an intermediately trimmed strong law of large numbers  to hold and  gave some illuminating additional examples. In this context see also \cite{MR1117267}, \cite{MR2085979}, and for complementary results concerning the largest summands we refer to \cite{mason_laws_1982}.

Our main theorem provides us with an explicit method to find such a trimming
sequence $\left(b_{n}\right)$. Further, in Remark \ref{ESnbn not ex} we show that the normalizing sequence $\left(d_{n}\right)_{n\in\mathbb{N}}$ 
is not necessarily asymptotic to the sequence of expectations $\left(\mathbb{E}\left(S_{n}^{b_{n}}\right)\right)_{n\in\mathbb{N}}$.

Finally, we get as a corollary that for non-negative random variables an intermediately trimmed strong law can be established even in a more general setting if we desist from constructing the normalizing sequence explicitly from $F$, see Corollary \ref{cor: thm A}. 
However, we would like to point out that this corollary could also be established differently using a quantile approach similar to the methods used by Haeusler in \cite{haeuler_asymptotic_1991}.

Following common notation we define the   ceiling of  $x\in \mathbb{R}$  as $\left\lceil x\right\rceil \coloneqq\min\left\{ n\in\mathbb{N}\colon n\geq x\right\} $ and  the floor of $x$ as   $\left\lfloor x\right\rfloor \coloneqq\max\left\{ n\in\mathbb{N}\colon n\leq x\right\} $.

\begin{thm}
\label{find bn} 
Let $\left(X_n\right)$ be sequence of non-negative random variables with infinite mean and distribution function $F$.
Further let $\left(t_{n}\right)_{n\in\mathbb{N}}$ be a sequence
of positive real numbers tending to infinity such that $F^{\leftarrow}\left(F\left(t_{n}\right)\right)=t_{n}$, for all $n\in\mathbb{N}$.
Fix $0<\epsilon<1/4$ 
such that for 
\[
a_{n}\coloneqq n\cdot\left(1-F\left(t_{n}\right)\right),\;\; d_{n}\coloneqq n\int_{0}^{t_{n}}x\,\mathrm{d}F\left(x\right),\,\, n\in\mathbb{N},
\]
we have 
\begin{align}
\lim_{n\to\infty}t_{n}/d_{n}\cdot\max\left\{a_{n}^{1/2+\epsilon}\left(\log\log n\right)^{1/2-\epsilon},\log n\right\}& =0.\label{cond 1 find bn}
\end{align}
Then with $b_{n}\coloneqq\left\lceil a_{n}+9\max\left\{a_n^{1/2+\epsilon}\cdot\left(\log\log n\right)^{1/2-\epsilon},\log\log n\right\}\right\rceil $,
$n\in\mathbb{N}$, we have 
\[
\lim_{n\rightarrow\infty}\frac{S_{n}^{b_{n}}}{d_{n}}=1\text{ a.s.}
\]
\end{thm}
Next we will show in an explicit  example that it is possible to apply Theorem \ref{find bn} also to rather involved  distribution functions.
\begin{ex}\label{ex: 1/logx}
Let  $F \coloneqq \sum_{j\in\mathbb{N}} (1-1/{j^2})\cdot\1_{J_j }$ with $J_j\coloneqq \left[2^{\left(j-1\right)^2},2^{j^2}\right)$.
Then, for $s_n\coloneqq2^{n^2}$ we have  $F^{\leftarrow}\left(F\left(s_n\right)\right)=s_n$,
\begin{align*}
 \int_0^{s_n}x\mathrm{d}F\left(x\right)&=\sum_{k=1}^{n}\left(\frac{1}{k^2}-\frac{1}{\left(k+1\right)^2}\right)\cdot 2^{k^2}
 \sim\sum_{k=1}^{n}\frac{2^{k^2}}{k^3}
\end{align*} 
 and
\begin{align*}
 \frac{s_n}{\left({\log_2 s_n}\right)^{3/2}}&\leq \sum_{k=1}^{n}\frac{2^{k^2}}{k^3} \leq \frac{s_n}{n^3}+ \sum_{k=1}^{n-1}\frac{2^{k^2}}{k^3}\sim \frac{s_n}{\left({\log_2 s_n}\right)^{3/2}}.
\end{align*}
Hence we obtain with respect to condition \eqref{cond 1 find bn}  that
\begin{align*}
 \frac{s_n}{ n\int_{0}^{s_{n}}x\,\mathrm{d}F\left(x\right)}\cdot a_n^{1/2+\epsilon}\cdot\left(\log\log n\right)^{1/2-\epsilon}
 &\sim \frac{\left(\log_2 s_n\right)^{3/2}}{n}\cdot \left(\frac{n}{\log_2 s_n}\right)^{1/2+\epsilon}\cdot\left(\log\log n\right)^{1/2-\epsilon}\\
 &=\frac{\left(\log_2 s_n\right)^{1-\epsilon}}{n^{1/2-\epsilon}}\cdot \left(\log\log n\right)^{1/2-\epsilon}
\end{align*}
tends to zero 
if, for some $\epsilon>0$ small,  we  choose 
 $t_n\coloneqq s_{\left\lfloor n^{1/4-\epsilon/2}\right\rfloor}= 2^{\left\lfloor n^{1/4-\epsilon/2}\right\rfloor^2}$.
This choice of $\left(t_n\right)$ also fulfills the condition that
\begin{align*}
 \frac{t_n}{d_n}\cdot \log n\asymp \frac{\left(\log_2 t_n\right)^{3/2}\cdot \log n}{n}
\end{align*}
tends to zero. 
We obtain that 
$a_n=n/\left\lfloor n^{1/4-\epsilon/2}\right\rfloor^2$. Thus, we can choose
\begin{align*}
 b_n\coloneqq\frac{n}{\left\lfloor n^{1/4-\epsilon/2}\right\rfloor^2}+9\cdot \left(\frac{n}{\left\lfloor n^{1/4-\epsilon/2}\right\rfloor^2}\right)^{1/2+\epsilon}\left(\log\log n\right)^{1/2-\epsilon}.
\end{align*}
\end{ex}
\begin{rem}
\label{ESnbn not ex}In general, for the normalizing sequence we do not have  $d_{n}\sim \mathbb{E}\left(S_{n}^{b_{n}}\right)$  as the following example shows.
Let the distribution function $F$ of $X_{1}$ be such that $F\left(x\right)=1-1/\log x$,
for all $x$ sufficiently large. Then, for all $n\in\mathbb{N}$, we have  $\mathbb{E}\left(S_{n}^{b_{n}}\right)=\infty$
 since  
\begin{align*}
\mathbb{E}\left(S_{n}^{b_{n}}\right)&\geq\mathbb{P}\left(S_{n}^{b_{n}}>x\right)\cdot x\geq\mathbb{P}\left(X_{i}>\frac{x}{n-b_{n}}\text{ for all }i\leq n\right)\cdot x\\
& \geq\left(\frac{1}{\log x-\log\left(n-b_{n}\right)}\right)^{n}\cdot x\geq\frac{x}{\left(\log x\right)^{n}}\to \infty,\;\;\mbox{for } x\to\infty .
\end{align*}
\end{rem}
A more general situation will be considered in Theorem \ref{Sbn}
and Corollary \ref{Sbn cont} giving further possibilities for finding
an appropriate trimming sequence $\left(b_{n}\right)$ and a corresponding
normalizing sequence $\left(d_{n}\right)$.

\begin{cor}
\label{cor: thm A} For a sequence of non-negative random variables
 $\left(X_{n}\right)$ with infinite mean
there exists a sequence of natural numbers $\left(b_{n}\right)_{n\in\mathbb{N}}$
with $b_{n}=o\left(n\right)$ and a sequence of positive reals $\left(d_{n}\right)_{n\in\mathbb{N}}$
such that 
\[
\lim_{n\rightarrow\infty}\frac{S_{n}^{b_{n}}}{d_{n}}=1\text{ a.s.}
\]
\end{cor}

\section{Moderate trimming for general distribution functions\label{gen trimming}}
Before stating this theorem we will
need further notation and  definitions.
Set 
\[
\Psi\coloneqq\left\{ u:\mathbb{N}\rightarrow\mathbb{R}^{+}\colon\sum_{n=1}^{\infty}\frac{1}{u\left(n\right)}<\infty\right\} .
\]
Further, set 
\begin{equation}
c\left(k,n\right)\coloneqq c_{\epsilon,\psi}\left(k,n\right)\coloneqq 8\left(\max\left\{ k,\log\psi\left(\left\lfloor \log n\right\rfloor \right)\right\} \right)^{1/2+\epsilon}\cdot\left(\log\psi\left(\left\lfloor \log n\right\rfloor \right)\right)^{1/2-\epsilon}\label{c(n)}
\end{equation}
for $k\in\mathbb{R}_{\geq1}$, $n\in\mathbb{N}_{\geq3}$,  $0<\epsilon<1/4$,
and $\psi\in\Psi$.
Furthermore, let us introduce the notation
$\check{F}\left(a\right)\coloneqq\lim_{x\nearrow a}F\left(x\right)$, 
$a\in\mathbb{R}$,  to denote 
the left-sided limit of $F$ in $a$.
\begin{thm}
\label{Sbn} 
Let $\left(X_n\right)$ be sequence of non-negative random variables with infinite mean and distribution function $F$.
Further,
let $\left(b_{n}\right)_{n\in\mathbb{N}}$ be a sequence
of natural numbers tending to infinity with $b_{n}=o\left(n\right)$
and $\left(t_{n}\right)_{n\in\mathbb{N}}$ a sequence of positive
real numbers tending to infinity such that $F^{\leftarrow}\left(F\left(t_{n}\right)\right)=t_{n}$,
for all $n\in\mathbb{N}$. For $n\in\mathbb{N}$ set 
\begin{equation*}
a_{n}^{+}\coloneqq n\cdot\left(1-\check{F}\left(t_{n}\right)\right),\; a_{n}^{-}\coloneqq n\cdot\left(1-F\left(t_{n}\right)\right),\;
d_{n}  \coloneqq n\int_{0}^{t_{n}}x\mathrm{d}F\left(x\right).\label{an+ an- def}
\end{equation*}
If there exist $0<\epsilon<1/4$ and $\psi,\widetilde{\psi}\in\Psi$ such that
\begin{align}
b_{n} & \geq a_{n}^{-}+c_{\epsilon,\psi}\left(a_{n}^{-},n\right)\label{an1}
\end{align}
and with $
\gamma_{n}\coloneqq\max\left\{ b_{n}-a_{n}^{-},b_{n}-a_{n}^{+}+c_{\epsilon,\psi}\left(a_{n}^{+},n\right)\right\} $ we have
\begin{align}
\lim _{n\to \infty}{t_{n}}/{d_{n}} \max\left\{\gamma_{n},{\log\widetilde{\psi}\left(n\right)}\right\} & =0,\label{cond ab}
\end{align}
 then   
\begin{align*}
\lim_{n\rightarrow\infty}\frac{S_{n}^{b_{n}}}{d_{n}}=1 & \text{ a.s.}
\end{align*}
\end{thm}
Under additional continuity assumptions on the distribution function the conditions simplify. 
\begin{cor}
\label{Sbn cont} In the above setting let us assume there exists $\kappa\in\mathbb{R}$ such that $F\lvert_{\left[\kappa,\infty\right)}$
is continuous. Let $\left(b_{n}\right)_{n\in\mathbb{N}}$
be a sequence of natural numbers tending to infinity with $b_{n}=o\left(n\right)$,
let $\left(t_{n}\right)_{n\in\mathbb{N}}$ be a sequence of positive
real numbers tending to infinity and set $a_{n}\coloneqq n\left(1-F\left(t_{n}\right)\right)$ and let $d_n$  be defined  as in Theorem \ref{Sbn} .
If there exist $0<\epsilon<1/4$ and $\psi,\widetilde{\psi}\in\Psi$ such that
\begin{align}
\tilde{\gamma}_{n}\coloneqq b_{n}-a_{n}\geq c_{\epsilon,\psi}\left(a_{n},n\right)\label{an 1}
\end{align}
and 
\begin{align}
\lim_{n\to \infty}{t_{n}}/{d_n }\max\left\{ \tilde{\gamma}_{n},\log\widetilde{\psi}\left(n\right)\right\} =0,\label{cond ab1}
\end{align}
then  
\begin{align*}
\lim_{n\rightarrow\infty}\frac{S_{n}^{b_{n}}}{d_{n}}=1\text{ a.s.}
\end{align*}
\end{cor}
In the following section we will first give a proof of Corollary \ref{Sbn cont}, Theorem \ref{find bn}, and Corollary \ref{cor: thm A} as a consequence of Theorem \ref{Sbn}. 
Finally, in Section \ref{subsec: proof gen thm}, we give the proof of the more general Theorem
\ref{Sbn}.
\subsection{Proofs of main theorems}
\begin{proof} [Proof of Corollary \ref{Sbn cont}] We will make
use of Theorem \ref{Sbn} as follows. First note that, since $F\lvert_{\left[\kappa,\infty\right)}$
is continuous, $a_{n}^{-}=a_{n}=a_{n}^{+}$, for $n$ sufficiently
large. Hence, \eqref{an 1} implies \eqref{an1} and the definition
of $\gamma_{n}$ in Theorem \ref{Sbn} implies $\tilde{\gamma}_{n}\leq\gamma_{n}\leq2\tilde{\gamma}_{n}$
for $n$ large. Combining the latter chain of inequalities with \eqref{cond ab1}
gives \eqref{cond ab} showing that all necessary conditions stated
in Theorem \ref{Sbn} are fulfilled.\end{proof} 

\begin{proof}[Proof of Theorem \ref{find bn}] 
Theorem \ref{find bn}
is a special case of Theorem \ref{Sbn}, where $a_{n}=a_{n}^{-}$ and up to a multiplicative constant we use  the fixed sequence $(\log\log n)$ instead of $c\left(a_n^-,n\right)$.
Choosing  $\psi\left(n\right)\coloneqq n^{9/8}$, we find by definition of $\left(b_{n}\right)$ in Theorem \ref{find bn} that  inequality
\eqref{an1} obviously holds.  
In the next steps we will prove that with $\widetilde{\psi}\left(n\right)\coloneqq n^2$ we have that  \eqref{cond 1 find bn} implies
\eqref{cond ab}:

On the one hand by definition of $b_n$ 
we have that 
\begin{align}
 b_n-a_n
 &\leq 9 \max\left\{a_n^{1/2+\epsilon}\cdot\left(\log\log n\right)^{1/2-\epsilon},\log n\right\}+1.\label{eq: proof TB 1}
\end{align}
On the other hand $a_n^+-c_{\epsilon,\psi}\left(a_n^+,n\right) \geq a_n-c_{\epsilon,\psi}\left(a_n,n\right)$ implies
\begin{align}
 b_{n}-a_{n}^{+}+c_{\epsilon,\psi}\left(a_{n}^{+},n\right)
 &\leq b_{n}-a_{n}+c_{\epsilon,\psi}\left(a_{n},n\right).\label{eq: proof TB 2}
\end{align}
Since 
\begin{align*}
 c_{\epsilon,\psi}\left(a_{n},n\right)
 &=8\left(\max\left\{ a_n,9/8\log\left\lfloor \log n\right\rfloor \right\} \right)^{1/2+\epsilon}\cdot\left(9/8\log\left\lfloor \log n\right\rfloor \right)^{1/2-\epsilon}\\
 &\leq 9\max\left\{a_n^{1/2+\epsilon}\cdot\left(\log\log n\right)^{1/2-\epsilon},\log n\right\},
\end{align*}
combining \eqref{eq: proof TB 1} and \eqref{eq: proof TB 2} yields
\begin{align*}
 b_{n}-a_{n}^{+}+c_{\epsilon,\psi}\left(a_{n}^{+},n\right)
 &\leq 18\max\left\{a_n^{1/2+\epsilon}\cdot\left(\log\log n\right)^{1/2-\epsilon},\log n\right\}.
\end{align*}

By our choice of $\widetilde{\psi}$ we have that 
\begin{align*}
 18\max\left\{a_n^{1/2+\epsilon}\cdot\left(\log\log n\right)^{1/2-\epsilon},\log n\right\}\geq\max\left\{\gamma_n, \log\widetilde{\psi}\left(n\right)\right\}
\end{align*}
and \eqref{cond 1 find bn} implies 
\eqref{cond ab}, which proves the statement of the theorem.
\end{proof}

\begin{proof}[Proof of Corollary \ref{cor: thm A}]
We aim to
apply Theorem \ref{find bn}. In order to do so we define, for $0<\epsilon<1/4$ and  $n\in \mathbb{N}$, 
\begin{align*}
t_{n}\coloneqq F^{\leftarrow}\left(F\left(n^{1/2-2\epsilon}\right)\right) \;\mbox{ and }\;d_{n}\coloneqq n\int_{0}^{t_{n}}x\,\mathrm{d}F\left(x\right).
\end{align*}
Since the expectation of $X_{1}$ is
infinite, $t_{n}$ tends to infinity and is thus a possible choice
to apply Theorem \ref{find bn}. We obtain thereby $a_{n}=n\left(1-F\left(n^{1/2-2\epsilon}\right)\right)$.
Furthermore,    by definition  we have  
$t_{n} =\inf\left\{ x\colon F\left(x\right)\geq F\left(n^{1/2-2\epsilon}\right)\right\} \leq n^{1/2-2\epsilon}$. 
Since $t_{n}$ tends to infinity, we also have  that $\lim_{n\rightarrow\infty}\int_{0}^{t_{n}}x\mathrm{d}F\left(x\right)=\infty$. Hence, to show condition \eqref{cond 1 find bn} it suffices to show
\begin{align*}
\lim_{n\to\infty}n^{1/2-2\epsilon}\cdot n^{-1}\cdot\max\left\{a_{n}^{1/2+\epsilon}\cdot\left(\log\log n\right)^{1/2-\epsilon},\log n\right\}& =0.
\end{align*}
Since 
$\lim_{n\to\infty} a_{n}^{1/2+\epsilon}\cdot\left(\log\log n\right)^{1/2-\epsilon}/n^{1/2+2\epsilon}=0$
and $\lim_{n\to\infty} \log n/n^{1/2+2\epsilon}=0$,
this follows immediately.
\end{proof}

\subsection{Proof of Theorem \ref{Sbn}}\label{subsec: proof gen thm}
In order to prove Theorem \ref{Sbn} we will need the following lemma.
\begin{lem}\label{log gamma log tilde gamma}
Let $a,b>1$ and $\psi\in\Psi$. Then there exists $\omega\in\Psi$ such that 
\begin{align}
\omega\left(\left\lfloor \log_b n\right\rfloor \right)\leq \psi\left(\left\lfloor \log_a n\right\rfloor\right).\label{phi psi a b} 
\end{align}
\end{lem}

\begin{proof}
We define $\omega:\mathbb{N}\rightarrow\mathbb{R}_{>0}$ by 
\begin{align}
\omega\left(n\right):=\min\left\{ \psi\left(\left\lfloor n\cdot {\log_a b}\right\rfloor +j\right)\colon j\in\left\{ 0,\ldots,\left\lceil {\log_a b}\right\rceil \right\} \right\}.\label{phi n}
\end{align}
Since $\psi\in\Psi$,  the functions 
$\overline{\psi}:\mathbb{N}\to\mathbb{R}_{>0}$ and $\widetilde{\psi}:\mathbb{N}\to\mathbb{R}_{>0}$ given by $\overline{\psi}\left(n\right):=\psi\left(\left\lfloor \kappa\cdot n\right\rfloor \right)$
with $\kappa>0$ and $\widetilde{\psi}\left(n\right):=\min\left\{ \psi\left(n\right),\ldots,\psi\left(n+k\right)\right\} $,  $k\in \mathbb{N}$, 
we have $\widetilde{\psi},\overline{\psi}\in\Psi$.
Hence, $\omega\in\Psi$.
Plugging  $\left\lfloor \log_b n\right\rfloor $ into $\omega$ in \eqref{phi n} yields 
\begin{align*}
\omega\left(\left\lfloor \log_b n\right\rfloor \right)=\min\left\{ \psi\left(\left\lfloor \left\lfloor {\log_b n}\right\rfloor\cdot{\log_a b }\right\rfloor +j\right)\colon j\in\left\{ 0,\ldots,\left\lceil {\log_a b}\right\rceil \right\} \right\} .
\end{align*}
Now, 
\eqref{phi psi a b} follows by observing
$\left\lfloor {\log_a n}\right\rfloor -\left\lceil {\log_a b}\right\rceil \leq\left\lfloor \left\lfloor {\log_b n}\right\rfloor\cdot{\log_a b }\right\rfloor \leq\left\lfloor {\log_a n}\right\rfloor$.
\end{proof}

Before proving Theorem \ref{Sbn} we will first prove the following
theorem concerning the truncated random variables defined as follows.
For a real valued sequence $\left(t_{n}\right)_{n\in\mathbb{N}}$
we let  
\begin{align}
  T_{n}^{t_{n}}\coloneqq\sum_{k=1}^{n}X_{k}\cdot\1_{\{X_{k}\leq t_{n} \}}\label{eq: def Tn tn}
\end{align}
denote the corresponding truncated sum process. In Theorem \ref{Thm: Sn* allg}
we will provide conditions on the real
valued sequence $\left(t_{n}\right)_{n\in\mathbb{N}}$ such that a
non-trivial strong law holds for $T_{n}^{t_{n}}$.

\begin{thm}
\label{Thm: Sn* allg} 
Let $\left(X_n\right)$ be sequence of non-negative random variables with infinite mean and distribution function $F$.
For a positive valued sequence $\left(t_{n}\right)_{n\in\mathbb{N}}$
assume $F\left(t_{n}\right)>0$ for all $n\in\mathbb{N}$ and there
exists $\psi\in\Psi$ such that 
\begin{align}
\frac{t_{n}}{\int_{0}^{t_{n}}x\,\mathrm{d}F\left(x\right)}=o\left(\frac{n}{\log\psi\left(n\right)}\right)\label{cond a}
\end{align}
holds. Then 
\begin{align*}
\lim_{n\rightarrow\infty}\frac{T_{n}^{t_{n}}}{n\int_{0}^{t_{n}}x\,\mathrm{d}F\left(x\right)}=1\text{ a.s.}
\end{align*}

\end{thm}
As an essential tool in the proof of Theorem \ref{Thm: Sn* allg}
we will need the following lemma which generalizes \emph{Bernstein's
inequality} and can be found for example in \cite{hoeffding_probability_1963}.
For the following we denote by $\mathbb{V}\left(\xi\right)$ the variance of a random variable $\xi$.
\begin{lem}
[Generalized Bernstein's inequality]\label{Bernstein} For $n\in\mathbb{N}$
let $Y_{1},\ldots,Y_{n}$ be independent random variables such that
$\left| Y_{i}-\mathbb{E}\left(Y_{i}\right)\right|\leq M<\infty$
for $i=1,\ldots,n$. Let $Z_{n}\coloneqq\sum_{i=1}^{n}Y_{i}$. Then
we have for all $t>0$ that 
\begin{align*}
\MoveEqLeft\mathbb{P}\left(\max_{k\leq n}\left|Z_{k}-\mathbb{E}\left(Z_{k}\right)\right|\geq t\right)\leq2\exp\left(-\frac{t^{2}}{2\mathbb{V}\left(Z_{n}\right)+\frac{2}{3}Mt}\right).
\end{align*}

\end{lem}
From Lemma \ref{Bernstein} the following lemma is easily deduced.
\begin{lem}
\label{Bernstein 1} For $n\in\mathbb{N}$ let $Y_{1},\ldots,Y_{n}$
be i.i.d.\ non-negative random variables such that $Y_{1}\leq K<\infty$.
Let $Z_{n}\coloneqq\sum_{i=1}^{n}Y_{i}$. Then we have for all $\kappa>0$
that 
\begin{align*}
\MoveEqLeft\mathbb{P}\left(\max_{k\leq n}\left|Z_{k}-\mathbb{E}\left(Z_{k}\right)\right|\geq\kappa\cdot\mathbb{E}\left(Z_{n}\right)\right)\leq2\exp\left(-\frac{3\kappa^{2}}{6+2\kappa}\cdot\frac{\mathbb{E}\left(Z_{n}\right)}{K}\right).
\end{align*}
\end{lem}

\begin{proof} [Proof of Theorem \ref{Thm: Sn* allg}] Fix $\epsilon>0$.
Clearly, we have $X_{1}^{t_{n}}\leq t_{n}$. Hence, we can apply Lemma
\ref{Bernstein 1} to the sequence $X_{1}^{t_{n}},\ldots,X_{n}^{t_{n}}$
to obtain 
\begin{align*}
\mathbb{P}\left(\left|T_{n}^{t_{n}}-\mathbb{E}\left(T_{n}^{t_{n}}\right)\right|\geq\epsilon\mathbb{E}\left(T_{n}^{t_{n}}\right)\right) & <2\exp\left(-\frac{3\epsilon^{2}}{6+2\epsilon}\cdot\frac{\mathbb{E}\left(T_{n}^{t_{n}}\right)}{t_{n}}\right).
\end{align*}

In order to apply the Borel-Cantelli Lemma note that condition \eqref{cond a}
implies there exists $\psi\in\Psi$ such that
\[
\frac{3\epsilon^{2}}{6+2\epsilon}\cdot\frac{n\cdot\int_{0}^{t_{n}}x\mathrm{d}F\left(x\right)}{t_{n}}\geq\log\psi\left(n\right),
\]
for all $n$ sufficiently large. Using the
fact that $\mathbb{E}\left(T_{n}^{t_{n}}\right)=n\int_{0}^{t_{n}}x\mathrm{d}F\left(x\right)$
gives 
\[
\sum_{n=1}^{\infty}\exp\left(-\frac{3\epsilon^{2}}{6+2\epsilon}\cdot\frac{\mathbb{E}\left(T_{n}^{t_{n}}\right)}{t_{n}}\right)<\infty
\]
and hence $\mathbb{P}\left(\left|T_{n}^{t_{n}}-\mathbb{E}\left(T_{n}^{t_{n}}\right)\right|\geq\epsilon\mathbb{E}\left(T_{n}^{t_{n}}\right)\text{infinitely often}\right)=0$.
Since $\epsilon>0$
is arbitrary, it follows that $\left|T_{n}^{t_{n}}-\mathbb{E}\left(T_{n}^{t_{n}}\right)\right|=o\left(\mathbb{E}\left(T_{n}^{t_{n}}\right)\right)$
almost surely and hence the assertion follows. \end{proof} 

In the proof of Theorem \ref{Sbn} we will use the following lemma which requires the probability space
 $\left(\Omega, \mathcal{A},\mathbb{P}\right)\coloneqq \bigotimes_{n=1}^{\infty}\left(\left[0,1\right], \mathcal{B},\lambda\right)$, where $\mathcal{B}$ denotes the Borel sets on $\left[0,1\right]$ and the family of events  
$\Lambda_{n,k}\coloneqq \pi_k^{-1}\left(\left[0,p_n\right]\right)$, for $k=1,\ldots,n\in\mathbb{N}$ with $\pi_k$ denoting the projection to the $k$th component.
\begin{lem}
\label{Philipp}
Let $A_{n,k}\coloneqq \mathbbm{1}_{\Lambda_{n,k}}$ for $k\leq n \in \mathbb{N}$.
Furthermore, let $c_{\epsilon,\psi}$ be defined as in \eqref{c(n)}
for $0<\epsilon<1/4$ and $\psi\in\Psi$. Then 
\begin{align*}
\mathbb{P}\left(\left|p_{n}\cdot n-\sum_{k=1}^{n}A_{n,k}\right|\geq c_{\epsilon,\psi}\left(p_{n}\cdot n,n\right)\text{ infinitely often}\right)=0.
\end{align*}

\end{lem}
\begin{proof} For $n\in \mathbb{N}$, we define $I_{n}\coloneqq\left[2^{n},2^{n+1}-1\right]\cap \mathbb{N}$,
$\kappa_{n}\coloneqq\left\lfloor \min_{k\in I_{n}}\log\psi\left(\left\lfloor \log k\right\rfloor \right)\right\rfloor $,
and $\rho:\mathbb{N}\rightarrow\mathbb{N}$ with $\rho\left(n\right)\coloneqq\left\lfloor \log_{2}n\right\rfloor $.
By this definition we have that $j\in I_{\rho\left(j\right)}=\left[2^{\rho\left(j\right)},2^{\rho\left(j\right)+1}-1\right]$.
In the following we will separately show for $c=c_{\epsilon,\psi}$ that 
\begin{align}
\mathbb{P}\left(\bigcap_{n\in\mathbb{N}}\bigcup_{\left\{ j\geq n\colon p_{j}\geq\kappa_{\rho\left(j\right)}/j\right\} }\left\{ \left|p_{j}\cdot j-\sum_{k=1}^{j}A_{j,k}\right|\geq c\left(p_{j}\cdot j,j\right)\right\} \right)=0\label{qn}
\end{align}
and 
\begin{align}
\mathbb{P}\left(\bigcap_{n\in\mathbb{N}}\bigcup_{\left\{ j\geq n\colon p_{j}\leq\kappa_{\rho\left(j\right)}/j\right\} }\left\{ \left|p_{j}\cdot j-\sum_{k=1}^{j}A_{j,k}\right|\geq c\left(p_{j}\cdot j,j\right)\right\} \right)=0.\label{bar qn}
\end{align}
Combining these two observations would then yield the statement of the lemma.

Since $c\left(p_j\cdot j,j\right)\geq \kappa_n$ for $j\in I_n$ we immediately have for all $j\in I_n$ in case that $\kappa_n\geq 2^{n+1}$ that 
\begin{align*}
\left\{\left|p_{j}\cdot j-\sum_{k=1}^{j}A_{j,k}\right|\geq c_{\epsilon,\psi}\left(p_{j}\cdot j,j\right) \right\}=\varnothing.
\end{align*}
 Hence, in order to show \eqref{qn} we only have to consider those $n\in\mathbb{N}$ for which $\kappa_n<2^{n+1}$. For $k\leq2^{n+1}$ and $\kappa_{n}\leq l\leq2^{n+1}$ set $B_{n,k}^l\coloneqq \mathbbm{1}_{\pi_k^{-1}\left(\left[0,l/2^{n+1}\right]\right)}$. 
The following properties are obviously fulfilled.
\begin{enumerate}[(a)]
\item $B_{n,k}^{l}$ is Bernoulli distributed with success probability $l/2^{n+1}$, 
\item \label{2} if $l/2^{n+1}\leq p_{j}<\left(l+1\right)/2^{n+1}$ for
 $j\in I_{n}$, then $B_{n,k}^{l}\leq A_{j,k}<B_{n,k}^{l+1}$
for all $k\leq j$, and 
\item The random variables $B_{n,1}^{l},\ldots,B_{n,2^{n+1}}^{l}$ are independent for
all fixed $l,n\in\mathbb{N}$. 
\end{enumerate}

Since $B_{n,k}^{l}$ is Bernoulli
distributed, we have that 
\begin{align*}
\mathbb{V}\left(\sum_{k=1}^{j}B_{n,k}^{l}\right) & =j\cdot\frac{l}{2^{n+1}}\left(1-\frac{l}{2^{n+1}}\right)\leq j\cdot\frac{l}{2^{n+1}}<l.
\end{align*}
Furthermore, $\left|B_{n,k}^{l}-\mathbb{E}\left(B_{n,k}^{l}\right)\right|\leq1$.
With these considerations we can apply Lemma \ref{Bernstein} to the
sequence $B_{n,1}^{l},\ldots,B_{n,2^{n+1}}^{l}$ and estimate

\begin{align}
\mathbb{P}\left(\max_{j\in I_{n}}\left|\frac{l}{2^{n+1}}\cdot j-\sum_{k=1}^{j}B_{n,k}^{l}\right|\geq\frac{3}{8}{ \min_{r\in I_n}}c\left(l,{ r}\right)\right)
 & \leq\mathbb{P}\left(\max_{j\in I_{n}}\left|\frac{l}{2^{n+1}}\cdot j-\sum_{k=1}^{j}B_{n,k}^{l}\right|\geq3l^{1/2+\epsilon}\kappa_n^{1/2-\epsilon}\right)\notag\\
 & <2\exp\left(-\frac{\left(3l^{1/2+\epsilon}\kappa_n^{1/2-\epsilon}\right)^{2}}{2l+\frac{2}{3}\cdot3l^{1/2+\epsilon}\kappa_n^{1/2-\epsilon}}\right).\label{kn l}
\end{align}
Since $l\geq\kappa_{n}$, we have that 
\begin{align*}
2l+\frac{2}{3}\cdot3l^{1/2+\epsilon}\kappa_n^{1/2-\epsilon}\leq4l
\end{align*}
and hence we can conclude from \eqref{kn l} that 
\begin{align*}
\mathbb{P}\left(\max_{j\in I_{n}}\left|\frac{l}{2^{n+1}}\cdot j-\sum_{k=1}^{j}B_{n,k}^{l}\right|\geq \frac{3}{8} \min_{r\in I_n}c\left(l,{ r}\right)\right)
  \leq2\exp\left(-\frac{9l^{1+2\epsilon}\kappa_n^{1-2\epsilon}}{4l}\right)
  <2\exp\left(-l^{2\epsilon}\kappa_n^{1-2\epsilon}\right).
\end{align*}
Furthermore, we estimate 
\begin{align}
\sum_{l=\kappa_n}^{2^{n+1}}\mathbb{P}&\left(\max_{j\in I_{n}}\left|\frac{l}{2^{n+1}}\cdot j-\sum_{k=1}^{j}B_{n,k}^{l}\right|>\frac{3}{8}{ \min_{r\in I_n}}c\left(l,{ r}\right)\right)\notag\\
  &<\sum_{l=\kappa_n}^{2^{n+1}}2\exp\left(-l^{2\epsilon}\kappa_n^{1-2\epsilon}\right)
  =2\exp\left(-\kappa_n\right)\sum_{l=\kappa_n}^{2^{n+1}}\exp\left(\kappa_n^{1-2\epsilon}\left(\kappa_n^{2\epsilon}-l^{2\epsilon}\right)\right).\label{qn 2^n}
\end{align}
Since $\epsilon<1/4$, we have that $x^{2\epsilon}$ is concave as
a function in $x$. Thus, we can estimate 
\begin{align}
\kappa_n^{1-2\epsilon}\left(l^{2\epsilon}-\kappa_n^{2\epsilon}\right) & \geq\kappa_n^{1-2\epsilon}\cdot2\epsilon\cdot l^{2\epsilon-1}\left(l-\kappa_n\right)\geq2\epsilon\left(l-\kappa_n\right)^{2\epsilon}\label{kappa epsilon}
\end{align}
if $2\leq l-\kappa_n=:h$. Define $H\coloneqq H_{\epsilon}\coloneqq\min\left\{ k\in\mathbb{N}_{\geq2}\colon\epsilon k^{2\epsilon}\geq\log k\right\} $.
Note that for all $k\geq H$ we also have $\epsilon k^{2\epsilon}\geq\log k$.
Using \eqref{kappa epsilon} and the definition of $H$ we obtain
\begin{align}
\sum_{l=\kappa_n+H}^{2^{n+1}}\exp\left(\kappa_n^{1-2\epsilon}\left(\kappa_n^{2\epsilon}-l^{2\epsilon}\right)\right) & \leq\sum_{l=\kappa_n+H}^{2^{n+1}}\exp\left(-2\epsilon\left(l-\kappa_n\right)^{2\epsilon}\right)\notag
  \leq\sum_{h=H}^{2^{n+1}-\kappa_n}\exp\left(-2\epsilon h^{2\epsilon}\right)\notag\\&\leq\sum_{h=H}^{2^{n+1}-\kappa_n}\exp\left(-2\log h\right) <\sum_{h=H}^{\infty}\exp\left(-2\log h\right)\leq\frac{\pi^{2}}{6}.\label{pi/4}
\end{align}

Since $l\geq\kappa_n$, every summand $\exp\left(\kappa_n^{1-2\epsilon}\left(\kappa_n^{2\epsilon}-l^{2\epsilon}\right)\right)$
is less than or equal to $1$ and hence we have 
\begin{align}
\sum_{l=\kappa_n}^{\kappa_n+H-1}\exp\left(\kappa_n^{1-2\epsilon}\left(\kappa_n^{2\epsilon}-l^{2\epsilon}\right)\right)\leq H.\label{sum H}
\end{align}

Combining \eqref{pi/4} and \eqref{sum H} with \eqref{qn 2^n} and
applying the definition of $\kappa_n$ yields 
\begin{align}
\sum_{l=\kappa_{n}}^{2^{n+1}}\mathbb{P}\left(\max_{j\in I_{n}}\left|\frac{l}{2^{n+1}}\cdot j-\sum_{k=1}^{j}B_{n,k}^{l}\right|>\frac{3}{4}{ \min_{r\in I_n}}c\left(l,{ r}\right)\right)
 & <2\left(H+\frac{\pi^{2}}{6}\right)\exp\left(-\kappa_n\right).\label{H pi exp}
\end{align}
We can conclude from Lemma \ref{log gamma log tilde gamma} that there exists $\omega\in\Psi$ such that 
\begin{align*}
\kappa_n&
\geq \min_{j\in I_n}\log\psi\left(\left\lfloor\log j\right\rfloor\right)-1=\min_{j\in I_n}\log\frac{\psi\left(\left\lfloor\log j\right\rfloor\right)}{\mathrm{e}}\geq\min_{j\in I_n}\log\omega\left(\left\lfloor\log_2 j\right\rfloor\right)=\log\omega\left(n\right).
\end{align*}

Thus, with \eqref{H pi exp} we obtain 
\begin{align*}
\sum_{n=1}^{\infty}\sum_{l=\kappa_{n}}^{2^{n+1}}\mathbb{P}\left(\max_{j\in I_{n}}\left|\frac{l}{2^{n+1}}\cdot j-\sum_{k=1}^{j}B_{n,k}^{l}\right|>\frac{3}{8}\min_{r\in I_n}c\left(l,{ r}\right)\right)
 & <\sum_{n=1}^{\infty}2\left(H+\frac{\pi^{2}}{6}\right)\frac{1}{\omega\left(n\right)}<\infty.
\end{align*}
Hence, we can apply the Borel-Cantelli Lemma and obtain for 
\[
C_{n}\coloneqq\left\{ \left(l,m\right)\in\mathbb{N}\times\mathbb{N}\colon m\geq n,\kappa_{m}\leq l\leq2^{m+1}\right\} 
\]
that 
\[
\mathbb{P}\left(\bigcap_{n\in\mathbb{N}}\bigcup_{\left(l,m\right)\in C_{n}}\left\{ \max_{j\in I_{m}}\left|\frac{l}{2^{m+1}}\cdot j-\sum_{k=1}^{j}B_{m,k}^{l}\right|>\frac{3}{8} \min_{r\in I_n}c\left(l,{ r}\right)\right\} \right)=0.
\]
Since 
\[
\max_{j\in I_{m}}\left|\frac{l}{2^{m+1}}\cdot j-\sum_{k=1}^{j}B_{m,k}^{l}\right|\leq\frac{3}{8}\min_{r\in I_m}c\left(l,{ r}\right)
\]
implies 
\[
\left|\frac{l}{2^{\rho\left(j\right)+1}}\cdot j-\sum_{k=1}^{j}B_{\rho\left(j\right),k}^{l}\right|\leq\frac{3}{8}\min_{r\in I_{\rho\left(j\right)}}c\left(l,r\right),
\]
for $j\in I_{m}$, we have with 
\[
D_{n}\coloneqq\left\{ \left(l,j\right)\in\mathbb{N}\times\mathbb{N}\colon j\geq n,\kappa_{\rho\left(j\right)}\leq l\leq2^{\rho\left(j\right)+1}\right\} 
\]
that 
\begin{align}
\mathbb{P}\left(\bigcap_{n\in\mathbb{N}}\bigcup_{\left(l,j\right)\in D_{n}}\left\{ \left|\frac{l}{2^{\rho\left(j\right)+1}}\cdot j-\sum_{k=1}^{j}B_{\rho\left(j\right),k}^{l}\right|>\frac{3}{8}\min_{r\in I_{\rho\left(j\right)}}c\left(l,r\right)\right\} \right)=0.\label{l c}
\end{align}

For $j\in I_{n}$ and $p_{j}\geq\kappa_{\rho\left(j\right)}/j$, we
can find $l\in\left\{ \kappa_{n},\ldots,2^{n+1}\right\} $ such that
\begin{align}
\frac{l}{2^{n+1}}\leq p_{j}<\frac{l+1}{2^{n+1}}.\label{l pj}
\end{align}
Let us assume this inequality and 
\begin{align}
\frac{l}{2^{n+1}}\cdot j-\sum_{k=1}^{j}B_{n,k}^{l}\leq\frac{3}{8}\min_{r\in I_n}c\left(l,r\right)\label{l Bn 1}
\end{align}
holds. Then it follows by the definition of $B_{n,k}^{l}$  
for $j\in I_n$ that 
\begin{align}
\frac{l}{2^{n+1}}\cdot j-\sum_{k=1}^{j}A_{j,k}\leq\frac{3}{8}\min_{r\in I_n}c\left(l,r\right)\leq\frac{3}{8}c\left(l,j\right).\label{A c}
\end{align}
For $j\in I_n$, we can conclude from the second inequality of \eqref{l pj} that   $p_{j}\cdot j-l/2^{n+1}\cdot j\leq1$,
and from the first inequality of \eqref{l pj} that $l\leq p_j\cdot 2^{n+1}\leq 2\cdot p_j\cdot j$. Thus,
\eqref{A c} implies 
\begin{align*}
p_{j}\cdot j-\sum_{k=1}^{j}A_{j,k} & <\frac{3}{4}c\left(p_{j}\cdot j,j\right)+1
 <c\left(p_{j}\cdot j,j\right),
\end{align*}
for $j \in I_n$ with $n\in \mathbb{N}$ sufficiently large. Analogously to the situation above we
get for $j\in I_n$ and under the assumption that \eqref{l pj} and 
\begin{align}
\sum_{k=1}^{j}B_{n,k}^{l+1}-\frac{l+1}{2^{n+1}}\cdot j & <\frac{3}{4}\min_{r\in I_n}c\left(l+1,r\right)\leq\frac{3}{4}c\left(l+1,j\right)\label{l Bn 2}
\end{align}
hold that 
\begin{align}
\sum_{k=1}^{j}A_{j,k}-\frac{l+1}{2^{n+1}}\cdot j & <\frac{3}{4}c\left(l+1,j\right).\label{eq: sum Ajk}
\end{align}
We conclude from the first inequality of \eqref{l pj} that 
$\left(l+1\right)/2^{n+1}\cdot j-p_{j}\cdot j\leq \left(l+1\right)/2^{n+1}\cdot j-l/2^{n+1}\cdot j\leq1$.
The first inequality of \eqref{l pj} also gives $l+1\leq 2\cdot p_j\cdot j+1$.
Thus, \eqref{eq: sum Ajk} implies
\begin{align*}
\sum_{k=1}^{j}A_{j,k}-p_{j}\cdot j & <\frac{3}{4}c\left(p_{j}\cdot j,j\right)+2
 <c\left(p_{j}\cdot j,j\right),
\end{align*}
for $j \in I_n$ with $n\in \mathbb{N}$ sufficiently large. Hence, we have proved that under condition
\eqref{l pj} the inequalities \eqref{l Bn 1} and \eqref{l Bn 2}
imply 
\begin{align*}
\left|\sum_{k=1}^{j}A_{j,k}-p_{j}\cdot j\right|<c\left(p_{j}\cdot j,j\right).
\end{align*}
This combined with \eqref{l c} proves \eqref{qn}.

In the last steps we prove \eqref{bar qn}. In order to do so we define
the triangular array of random variables $\left(\bar{B}_{n,k}\right)$
with $k,n\in\mathbb{N}$, $k\leq2^{n+1}$ by $\overline{B}_{n,k}\coloneqq \mathbbm{1}_{\pi_k^{-1}\left(\left[0,\kappa_n/2^n\right]\right)}$.
It immediately follows that
\begin{enumerate}[(a)]
\item each $\bar{B}_{n,k}$ is Bernoulli distributed with success probability
$q_{n}\coloneqq\kappa_{n}/2^{n}$, 
\item \label{bar B 2} if $p_{j}\leq q_{n}$ for some $j\in I_{n}$, then
$\bar{B}_{n,k}\geq A_{j,k}$ for all $k\leq j$, and 
\item The random variables $\bar{B}_{n,1},\ldots,\bar{B}_{n,2^{n+1}}$ are independent
for all fixed $n\in\mathbb{N}$. 
\end{enumerate}
Since $\mathbb{E}\left(\bar{B}_{n,1}\right)=q_{n}$ we have
that 
\begin{align*}
\frac{3}{8}\cdot c\left(\kappa_{n},2^{n}\right)\geq\frac{3}{2}\cdot q_{n}\cdot2^{n+1}>\frac{3}{2}\cdot\mathbb{E}\left(\sum_{k=1}^{2^{n+1}-1}\bar{B}_{n,k}\right).
\end{align*}
Furthermore $\bar{B}_{n,k}\leq1$ and we can apply Lemma \ref{Bernstein 1}.
This yields 
\begin{align*}
\MoveEqLeft\mathbb{P}\left(\max_{j\in I_{n}}\left|q_{n}\cdot j-\sum_{k=1}^{j}\bar{B}_{n,k}\right|\geq\frac{3}{8}c\left(q_{n}\cdot2^{n},2^{n}\right)\right)\\
 & \leq\mathbb{P}\left(\max_{j\in I_{n}}\left|q_{n}\cdot j-\sum_{k=1}^{j}\bar{B}_{n,k}\right|\geq\frac{3}{2}\cdot\mathbb{E}\left(\sum_{k=1}^{2^{n+1}-1}\bar{B}_{n,k}\right)\right)\\
 & \leq2\exp\left(-\frac{3\left(\frac{3}{2}\right)^{2}}{6+2\cdot\frac{3}{2}}\cdot\mathbb{E}\left(\sum_{k=1}^{2^{n+1}-1}\bar{B}_{n,k}\right)\right)=2\exp\left(-\frac{3}{4}\cdot\mathbb{E}\left(\sum_{k=1}^{2^{n+1}-1}\bar{B}_{n,k}\right)\right)\\
 & <2\exp\left(-\frac{3}{4}\cdot q_{n}\cdot\left(2^{n+1}-1\right)\right)<2\exp\left(-\kappa_{n}\right),
\end{align*}
for all $n\in\mathbb{N}$. By the definition of $\kappa_{n}$ it follows that 
\begin{align*}
\mathbb{P}\left(\max_{j\in I_{n}}\left|q_{n}\cdot j-\sum_{k=1}^{j}\bar{B}_{n,k}\right|\geq\frac{3}{8}c\left(q_{n}\cdot2^{n},2^{n}\right)\right)
 & <2\exp\left(-\left\lfloor \min_{j\in I_{n}}\log\psi\left(\left\lfloor \log j\right\rfloor \right)\right\rfloor\right)\\
 & \leq\frac{2}{\min_{j\in I_{n}}\psi\left(\left\lfloor \log j\right\rfloor \right)}.
\end{align*}

With the considerations from above we have that $\psi\in\Psi$ implies
that $\tilde{\psi}:\mathbb{N}\rightarrow\mathbb{R}^{+}$ lies in $\Psi$
where $\tilde{\psi}\left(n\right)\coloneqq\min_{j\in I_{n}}\psi\left(\left\lfloor \log j\right\rfloor \right)$.

With these considerations we can apply the Borel-Cantelli Lemma and
obtain 
\begin{align}
\mathbb{P}\left(\max_{j\in I_{n}}\left|q_{n}\cdot j-\sum_{k=1}^{j}\bar{B}_{n,k}\right|\geq\frac{3}{4}c\left(q_{n}\cdot2^{n},2^{n}\right)\text{ infinitely often}\right)=0.\label{BC bar B}
\end{align}
For $p_{j}\leq q_{n}$ with $j\in I_{n}$ and thus in particular for
$p_{j}\leq\kappa_{\rho\left(j\right)}/j$ we have that $\bar{B}_{j,k}\geq A_{j,k}$
for $j\leq k$. Hence, 
$\sum_{k=1}^{j}\bar{B}_{j,k}-q_{n}\cdot j<3\cdot q_{n}\cdot2^{n}
$
implies 
$\sum_{k=1}^{j}A_{j,k}-q_{n}\cdot j <3\cdot q_{n}\cdot2^{n}$.
Adding $q_{n}\cdot j-p_{j}\cdot j$ 
\begin{align*}
\sum_{k=1}^{j}A_{j,k}-p_{j}\cdot j & <3\cdot q_{n}\cdot2^{n}+q_{n}\cdot j-p_{j}\cdot j
\leq4\cdot q_{n}\cdot j
\leq4\cdot\log\psi\left(\left\lfloor \log j\right\rfloor \right)
\end{align*}
Since $p_{j}\leq\kappa_{n}/2^{n}$ for $j\in I_{n}$, we have that
$p_{j}\cdot j\leq\log\psi\left(\left\lfloor \log j\right\rfloor \right)$
and thus $c\left(p_{j}\cdot j,j\right)=4\cdot\log\psi\left(\left\lfloor \log j\right\rfloor \right)$.
Hence, 
$\sum_{k=1}^{j}A_{j,k}-p_{j}\cdot j <c\left(p_{j}\cdot j,j\right)$.
Combining this with \eqref{BC bar B} and noting that on the other
hand we have for $p_{j}\leq q_{n}$ and $j\in I_{n}$ that $p_{j}\cdot j-\sum_{k=1}^{j}A_{j,k}\leq q_{n}\cdot2^{n+1}\leq c\left(p_{j}\cdot j,j\right)$
yields \eqref{bar qn}. \end{proof} 

\begin{proof} [Proof of Theorem \ref{Sbn}] Since $\left(t_{n}\right)$
fulfills condition \eqref{cond ab}, { it follows that $F\left(t_n\right)>0$, for $n$ sufficiently large and} we can apply Theorem \ref{Thm: Sn* allg}
which yields
\begin{align}
T_{n}^{t_{n}} & \sim n\int_{0}^{t_{n}}x\mathrm{d}F\left(x\right)\text{ a.s.}\label{Tn asymp}
\end{align}
We want to apply Lemma \ref{Philipp} to the random variables 
\begin{align*}
\left(\1_{\left\{ X_{1}>t_{n}\right\} },\ldots,\1_{\left\{ X_{n}>t_{n}\right\} }\right)\text{ and }\left(\mathbbm{1}_{\left\{ X_{1}\geq t_{n}\right\} },\ldots,\mathbbm{1}_{\left\{ X_{n}\geq t_{n}\right\} }\right).
\end{align*}
These random variables are defined on a different probability space, however the triangular scheme $\left(\1_{\left\{ X_{k}>t_{n}\right\}}\right)_{n\in\mathbb{N},k\leq n}$ 
is identically distributed to $\left(A_{n,k}\right)_{n\in\mathbb{N},k\leq n}$ defined in Lemma \ref{Philipp} if we set $p_n\coloneqq \mathbb{P}\left(X_{k}>t_{n}\right)$.
To see this we note that $A_{n_1,k_1},\ldots, A_{n_i,k_i}$ are independent by construction if all $k_j$, $j=1,\ldots,i$ are different.
The same is true for $\1_{\left\{ X_{k_1}>t_{n_1}\right\}},\ldots, \1_{\left\{ X_{k_i}>t_{n_i}\right\}}$ since $X_{k_1},\ldots, X_{k_i}$ are independent. 
To calculate the finite dimensional distribution for $k_1=\ldots=k_i$ we only have to consider 
$\mathbb{P}\left(A_{n,1}>0\right)$ which gives a probability strictly between $0$ and $1$. 
\begin{align*}
 \MoveEqLeft\mathbb{P}\left(A_{n_1,1}>0,\ldots, A_{n_i,1}>0\right)\\
 &=\mathbb{P}\left(x\in\left[0,p_{n_1}\right]\cap\ldots\cap\left[0,p_{n_i}\right]\right)
 =\mathbb{P}\left(x\in\left[0,\min_{1\leq j\leq i}p_{n_j}\right]\right)
 =\min_{1\leq j\leq i}p_{n_j}.
\end{align*}
On the other hand we have that
\begin{align*}
\mathbb{P}\left(\1_{\left\{ X_{1}>t_{n_1}\right\}}>0,\ldots, \1_{\left\{ X_{1}>t_{n_i}\right\}}>0\right)
 &=\mathbb{P}\left(X_1>\max_{1\leq j\leq i} t_{n_j}\right)=\min_{1\leq j\leq i}p_{n_j}.
 \end{align*}
Furthermore,
\begin{align*}
 \MoveEqLeft\mathbb{P}\left(\left|p_{n}\cdot n-\sum_{k=1}^{n}A_{n,k}\right|\geq c_{\epsilon,\psi}\left(p_{n}\cdot n,n\right)\text{ infinitely often}\right)\\
 &=\lim_{k\to\infty}\mathbb{P}\left(\bigcup_{n\geq k}\left\{\left|p_{n}\cdot n-\sum_{k=1}^{n}A_{n,k}\right|\geq c_{\epsilon,\psi}\left(p_{n}\cdot n,n\right)\right\}\right)
\end{align*}
and since 
\begin{align*}
\MoveEqLeft\mathbb{P}\left(\bigcup_{n\geq k}\left\{\left|p_{n}\cdot n-\sum_{k=1}^{n}A_{n,k}\right|\geq c_{\epsilon,\psi}\left(p_{n}\cdot n,n\right)\right\}\right)\\
&=\mathbb{P}\left(\bigcup_{n\geq k}\left\{\left|p_{n}\cdot n-\sum_{k=1}^{n}\mathbbm{1}_{\left\{X_k>t_n\right\}}\right|\geq c_{\epsilon,\psi}\left(p_{n}\cdot n,n\right)\right\}\right),
\end{align*}
for all $k\in\mathbb{N}$, it also follows that 
\begin{align*}
\mathbb{P}\left(\left|p_{n}\cdot n-\sum_{k=1}^{n}\mathbbm{1}_{\left\{X_k>t_n\right\}}\right|\geq c_{\epsilon,\psi}\left(p_{n}\cdot n,n\right)\text{ infinitely often}\right)=0.
\end{align*}
The argumentation for $\left(\1_{\left\{ X_{k}\geq t_{n}\right\}}\right)$ follows analogously.

We have that the success probability for $\mathbbm{1}_{\left\{ X_{1}>t_{n}\right\} }$
is $a_{n}^{-}/n$ and the success probability for $\mathbbm{1}_{\left\{ X_{1}\geq t_{n}\right\} }$
is $a_{n}^{+}/n$. Thus,
we have a.e.\ 
\begin{align}
a_{n}^{+}-c\left(a_{n}^{+},n\right)
 & \leq\#\left\{ i\leq n\colon X_{i}\geq t_{n}\right\}
 \leq a_{n}^{+}+c\left(a_{n}^{+},n\right)\text{ eventually}\label{a c -}
\end{align}
and a.e.\ 
\begin{align}
a_{n}^{-}-c\left(a_{n}^{-},n\right) & \leq\#\left\{ i\leq n\colon X_{i}>t_{n}\right\}
 \leq a_{n}^{-}+c\left(a_{n}^{-},n\right)\text{ eventually.}\label{a c +}
\end{align}
We can conclude from \eqref{a c +} that a.e. 
\begin{align}
S_{n}^{a_{n}^{-}+c\left(a_{n}^{-},n\right)}\leq T_{n}^{t_{n}}\text{ eventually.}\label{San Tn}
\end{align}

Combining \eqref{a c -} and \eqref{a c +} we have that a.e. 
\begin{align}
a_{n}^{+}-a_{n}^{-}-c\left(a_{n}^{+},n\right)-c\left(a_{n}^{-},n\right) & \leq\#\left\{ i\leq n\colon X_{i}=t_{n}\right\} \text{ eventually.}\label{a c +-}
\end{align}
We first consider the case that $b_{n}\leq a_{n}^{+}-c\left(a_{n}^{+},n\right)$.
Then 
\begin{align*}
b_{n}-a_{n}^{-}-c\left(a_{n}^{-},n\right)\leq a_{n}^{+}-a_{n}^{-}-c\left(a_{n}^{+},n\right)-c\left(a_{n}^{-},n\right).
\end{align*}
Hence, \eqref{an1} combined with \eqref{San Tn} and \eqref{a c +-}
yields a.e. 
\begin{align}
S_{n}^{b_{n}}\leq T_{n}^{t_{n}}-\left(b_{n}-a_{n}^{-}-c\left(a_{n}^{-},n\right)\right)t_{n}\text{ eventually.}\label{Sbn <}
\end{align}
On the other hand since $X_{n}^{t_{n}}\leq t_{n}$ it follows by
\eqref{a c +} that a.e. 
\begin{align*}
T_{n}^{t_{n}}-2c\left(a_{n}^{-},n\right)t_{n}\leq S_{n}^{a_{n}^{-}+c\left(a_{n}^{-},n\right)}\text{ eventually}
\end{align*}
and trimming the sum by $b_{n}-\left(a_{n}^{-}+c\left(a_{n}^{-},n\right)\right)$
more summands yields a.e. 
\begin{align}
\MoveEqLeft T_{n}^{t_{n}}-\left(b_{n}-a_{n}^{-}-c\left(a_{n}^{-},n\right)\right)t_{n}-2c\left(a_{n}^{-},n\right)t_{n}\notag\\
 & =T_{n}^{t_{n}}-\left(b_{n}-a_{n}^{-}+c\left(a_{n}^{-},n\right)\right)t_{n}\notag\\
 & \leq S_{n}^{b_{n}}\text{ eventually.}\label{Sbn >}
\end{align}
Combining \eqref{Sbn <} and \eqref{Sbn >} yields a.e. 
\begin{align*}
\left|S_{n}^{b_{n}}-T_{n}^{t_{n}}+\left(b_{n}-a_{n}^{-}\right)t_{n}\right| & \leq c\left(a_{n}^{-},n\right)\cdot t_{n}\leq\gamma_{n}\cdot t_{n}\text{ eventually.}
\end{align*}
Let us now consider the case $b_{n}>a_{n}^{+}-c\left(a_{n}^{+},n\right)$.
This implies 
\begin{align*}
b_{n}-a_{n}^{-}-c\left(a_{n}^{-},n\right)>a_{n}^{+}-a_{n}^{-}-c\left(a_{n}^{+},n\right)-c\left(a_{n}^{-},n\right)
\end{align*}
and with \eqref{San Tn} and \eqref{a c +-} it follows that a.e.
\begin{align}
S_{n}^{b_{n}}\leq T_{n}^{t_{n}}-\left(a_{n}^{+}-a_{n}^{-}-c\left(a_{n}^{+},n\right)-c\left(a_{n}^{-},n\right)\right)t_{n}\text{ eventually.}\label{Sbn <1}
\end{align}
On the other hand we have that \eqref{Sbn >} still holds if $b_{n}>a_{n}^{+}-c\left(a_{n}^{+},n\right)$.
Considering 
\begin{align*}
\left(b_{n}-a_{n}^{-}\right)-\left(a_{n}^{+}-a_{n}^{-}-c\left(a_{n}^{-},n\right)-c\left(a_{n}^{+},n\right)\right)
 =b_{n}-a_{n}^{+}+c\left(a_{n}^{+},n\right)+c\left(a_{n}^{-},n\right)
\end{align*}
and combining \eqref{Sbn >} and \eqref{Sbn <1} yields that a.e.
for $n$ sufficiently large 
\begin{align*}
\left|S_{n}^{b_{n}}-T_{n}^{t_{n}}+\left(b_{n}-a_{n}^{-}\right)t_{n}\right| & \leq b_{n}-a_{n}^{+}+c\left(a_{n}^{+},n\right)+c\left(a_{n}^{-},n\right)
\end{align*}
 and thus
\[
\left|S_{n}^{b_{n}}-T_{n}^{t_{n}}+\left(b_{n}-a_{n}^{-}\right)t_{n}\right|=\mathcal{O}\left(\gamma_{n}\cdot t_{n}\right)\text{\,{a.s.}}
\]

Combining this with \eqref{cond ab} and \eqref{Tn asymp} yields
the statement of the theorem. \end{proof}

\subsection*{Acknowledgements}
We thank David Mason for mentioning the publications on trimmed sums for slowly varying tails, Péter Kevei for useful comments on an earlier draft of this paper and the referee for his or her valuable comments and careful proofreading which considerably improved the quality of this paper.

\end{document}